\documentclass[12pt]{article}

\usepackage{booktabs}

\usepackage{graphicx}
\usepackage{nameref}
\usepackage[shortlabels,inline]{enumitem}
\usepackage{empheq}
\usepackage[shortlabels]{enumitem}
\setlist[enumerate]{nosep}
\usepackage[doc]{optional}
\usepackage{xcolor}
\usepackage[colorlinks=true,
            linkcolor=refkey,
            urlcolor=lblue,
            citecolor=red]{hyperref}
\usepackage{float}
\usepackage{soul}
\usepackage{graphicx}
\definecolor{labelkey}{rgb}{0,0.08,0.45}
\definecolor{refkey}{rgb}{0,0.6,0.0}
\definecolor{Brown}{rgb}{0.45,0.0,0.05}
\definecolor{lime}{rgb}{0.00,0.8,0.0}
\definecolor{lblue}{rgb}{0.5,0.5,0.99}

 \usepackage{mathpazo}

\colorlet{hlcyan}{cyan!30}

\usepackage{stmaryrd}
\usepackage{amssymb}

\makeatletter
\def\namedlabel#1#2{\begingroup
   \def\@currentlabel{#2}%
   \label{#1}\endgroup
}
\makeatother

\newcommand{\seppfour}{\setlength{\itemsep}{-4pt}}

\newcommand{\tsqrt}[1]{\ensuremath{\textstyle{\sqrt{#1}}}}

\newcommand{\nnn}{\ensuremath{{n\in{\mathbb N}}}}

\newcommand{\menge}[2]{\big\{{#1}~\big |~{#2}\big\}}

\newcommand{\Menge}[2]{\Big\{{#1}~\Big|~{#2}\Big\}}

\newcommand{\fenv}[1]%
{\ensuremath{\,\overrightarrow{\operatorname{env}}_{#1}}}
\newcommand{\benv}[1]%
{\ensuremath{\,\overleftarrow{\operatorname{env}}_{#1}}}

\newcommand{\scal}[2]{\left\langle{#1},{#2}  \right\rangle}
\newcommand{\tscal}[2]{\langle{#1},{#2}\rangle}
\newcommand{\bscal}[2]{\big\langle{#1},{#2}\big\rangle}

\newcommand{\RR}{\ensuremath{\mathbb R}}
\newcommand{\RP}{\ensuremath{\left[0,+\infty\right[}}
\newcommand{\RPP}{\ensuremath{\left]0,+\infty\right[}}
\newcommand{\RM}{\ensuremath{\mathbb{R}_-}}

\newcommand{\NN}{\ensuremath{\mathbb N}}

\newcommand{\dom}{\ensuremath{\operatorname{dom}}}

\newcommand{\conv}{\ensuremath{\operatorname{conv}\,}}

\newcommand{\cspan}{\ensuremath{\overline{\operatorname{span}}\,}}

\newcommand{\Id}{\ensuremath{\operatorname{Id}}}

\newcommand{\minf}{\ensuremath{-\infty}}
\newcommand{\pinf}{\ensuremath{+\infty}}

{\begin{list}{}{%
\settowidth{\labelwidth}{\textrm{#1~}}%
\setlength{\leftmargin}{\labelwidth+\labelsep}}}
{\end{list}}
\usepackage{amsthm}
\usepackage[capitalize,nameinlink]{cleveref}
\crefname{lemma}{Lemma}{Lemmas}
\crefname{equation}{}{equations}
\crefname{figure}{Figure}{Figures}
\crefname{chapter}{Appendix}{chapters}
\crefname{item}{}{items}
\crefname{enumi}{}{}
\newtheorem{theorem}{Theorem}[section]

\newtheorem{corollary}[theorem]{Corollary}

\newtheorem{proposition}[theorem]{Proposition}

\newtheorem{example}[theorem]{Example}
\newtheorem{ex}[theorem]{Example}
\newtheorem{fact}[theorem]{Fact}
\newtheorem{algo}[theorem]{Algorithm}
\newtheorem{remark}[theorem]{Remark}

\providecommand{\RR}{\mathbb{R}}

\providecommand{\conv}{\operatorname{conv}}

\providecommand{\cone}{\operatorname{cone}}
\providecommand{\clcone}{\overline{\operatorname{cone}}\,}

\providecommand{\dom}{\operatorname{dom}}

\providecommand{\Id}{\operatorname{{ Id}}}

\providecommand{\NN}{\mathbb{N}}

\providecommand{\rec}{\operatorname{rec}}
\providecommand{\barc}{\operatorname{bar}}
\providecommand{\Id}{\operatorname{Id}}

\newcommand{\cran}{\ensuremath{\overline{\operatorname{ran}}\,}}

\providecommand{\RR}{\mathbb{R}}
\providecommand{\NN}{\mathbb{N}}

\definecolor{myblue}{rgb}{.8, .8, 1}
  \newcommand*\mybluebox[1]{%
    \colorbox{myblue}{\hspace{1em}#1\hspace{1em}}}

\allowdisplaybreaks 

\begin{document}

\title{\textsc{
The homogenization cone:\\ polar cone and projection}}
\author{
Heinz H.\ Bauschke\thanks{
Mathematics, University
of British Columbia,
Kelowna, B.C.\ V1V~1V7, Canada. E-mail:
\texttt{heinz.bauschke@ubc.ca}.},~
Theo Bendit\thanks{
Mathematics, University
of British Columbia,
Kelowna, B.C.\ V1V~1V7, Canada. E-mail:
\texttt{theo.bendit@ubc.ca}.},~
~and 
Hansen Wang\thanks{
Mathematics, University
of British Columbia, 
Kelowna, B.C.\ V1V~1V7, Canada.
E-mail: \texttt{wanghansenwh@163.com}
}
}

\date{May 31, 2022}
\maketitle

\vskip 8mm

\begin{abstract} 
Let $C$ be a closed convex subset of a real Hilbert space
containing the origin, and assume that 
$K$ is the homogenization cone of $C$, i.e., 
the smallest closed convex cone containing $C\times\{1\}$.
Homogenization cones play an important role in optimization as they include,
for instance, the second-order/Lorentz/``ice cream'' cone. 

In this note, we discuss the polar cone of $K$ as well as an algorithm for finding
the projection onto $K$ provided that the projection onto $C$ is available.
Various examples illustrate our results.


\end{abstract}

{\small


}

\section{Introduction}

Throughout, 
\begin{empheq}[box=\mybluebox]{equation}
\label{e:X}
\text{$X$ is a real Hilbert space}
\end{empheq}
with inner product $\scal{\cdot}{\cdot}$ and induced norm $\|\cdot\|$.
We also assume that 
\begin{empheq}[box=\mybluebox]{equation}
\label{e:C}
\text{$C$ is a closed convex subset of $X$ such that $0\in C$.}
\end{empheq}
Recall that 
\begin{equation}
\cone(C\times\{1\})
:= \bigcup_{\rho>0}\rho(C\times\{1\})
\end{equation}
and denote the closure of this set by $\clcone(C\times\{1\})$.
Note that this notation is not entirely standard, some authors
allow $\rho=0$ in the definition of the conical hull. 

This note concerns the cone 
\begin{empheq}[box=\mybluebox]{equation}
\label{e:K}
K := \clcone(C\times\{1\}),
\end{empheq}
which is called the \emph{homogenization} (also known as \emph{conification}) of $C$; 
see, e.g, \cite{Brinkhuis} and \cite{RoshchinaTuncel}.

The purpose of this note is two-fold.
Starting with the representation (see also \cite[Theorem~8.2]{Rock70})
\begin{equation}
\label{e:rawsome}
K = 
\cone\big(C\times \{1\}\big) \uplus \big(\rec(C)\times \{0\}\big),
\end{equation}
where ``$\uplus$'' indicates that the union is disjoint, we first 
obtain the following beautiful formula for the polar cone 
\begin{equation}
\label{e:rawsomepo}
K^\ominus = 
\cone\big(C^\odot\times \{-1\}\big) \uplus \big(\rec(C^\odot) \times \{0\}\big),
\end{equation}
where $C^\odot$ is the polar set of $C$.

Secondly, we will discuss the computation of $P_K$ provided that $P_C$ is available.
In some cases, explicit formulas are available; in other cases, we are 
able to present an algorithm that allows us to find $P_K$ iteratively.

The remainder of this paper is organized as follows.
In \cref{sec:aux}, we collect basic properties and examples that are used
later in the paper. 
Formulas for the polar cone of $K$ are derived in \cref{sec:polar}. 
Scaled sets are investigated in \cref{sec:scaled}. 
In \cref{sec:proj}, we study the computation of $P_K$ provided that 
$P_C$ is available. 
The final \cref{sec:examples} contains a study of examples. 

The notation employed throughout the paper is standard and follows 
\cite{BC2017} and \cite{Rock70}.

\section{Auxiliary results and examples}
\label{sec:aux}

\subsection{Recession cone}

We recall that our standing assumptions are \cref{e:X} and \cref{e:C}, 
and that  the \emph{recession cone} of $C$,
denoted by $\rec(C)$, is defined by
\begin{equation}
\label{e:rec}
\rec(C)=\menge{x\in X}{x+C \subseteq C}. 
\end{equation}

The following facts concerning recession cones, 
which are largely part of the folklore, are stated here 
for completeness and the reader's convenience.

\begin{fact} 
\label{f:rec}
Let $x\in X$. 
The following statements hold:
\begin{enumerate}
\item 
\label{f:rec0}
$\rec(C)$ is a closed convex cone and $0\in \rec(C)$.
\item 
\label{f:rec0.5}
$x\in \rec(C)$
$\Leftrightarrow$
there exist $(\alpha_n)_\nnn$ in $\left]0,1\right]$ and 
$(c_n)_\nnn$ in $C$ such that $\alpha_n\to 0$ and $\alpha_nc_n\to x$. 
\item 
\label{f:rec1}
$x\in \rec(C)$
$\Leftrightarrow$
$\RP x\subseteq C$.
\item 
\label{f:rec2}
$\rec(C) = \bigcap_{\rho>0}\rho C$. 
\item 
\label{f:rec3}
$\rec(C)=\{0\}$
$\Leftrightarrow$
$C$ is \emph{linearly bounded}, i.e., 
$(\forall d\in X)$
$C\cap \RP \cdot d$ is bounded. 
\item 
\label{f:rec4}
If $X$ is finite-dimensional, then: 
$C$ is bounded $\Leftrightarrow$ $\rec(C) = \{0\}$. 
\item 
\label{f:rec5}
If $Y$ is a real Hilbert space and $D$ is a nonempty convex subset 
of $Y$, then $\rec(C\times D) = \rec(C)\times\rec(D)$.
\end{enumerate}
\end{fact}

\begin{proof}
\cref{f:rec0}:
This follows from \cite[Proposition~6.49(i)\&(v) and Proposition~6.24(ii)]{BC2017}. 

\cref{f:rec0.5}: See \cite[Proposition~6.51]{BC2017}. 

\cref{f:rec1}:
``$\Rightarrow$'':
Suppose that $x\in\rec(C)$ and let  $\rho\geq 0$. 
By \cref{f:rec0}, $\rho x\in \rec(C)$.
Because  $0\in C$, 
we now deduce from \cref{e:rec}
that $\rho x = \rho x+0\in C$. 
``$\Leftarrow$'':
Let $(\alpha_n)_\nnn$ be a sequence in $\left]0,1\right]$
such that $\alpha_n\to 0$.
By assumption, $(\forall\nnn)$ $c_n := x/\alpha_n \in C$.
Because $(\forall\nnn)$ $\alpha_n c_n = x$,
it follows from \cref{f:rec0.5} that $x\in \rec(C)$. 

\cref{f:rec2}:
``$\Rightarrow$'':
Suppose that $x\in\rec(C)$ and  let $\rho > 0$. 
By \cref{f:rec1}, $(1/\rho)x\in C$. 
Hence $x\in \rho C$.
``$\Leftarrow$'':
Suppose that $x\in \bigcap_{\rho>0}\rho C$ and let $\alpha \geq 0$.
If $\alpha=0$, then $\alpha x = 0 \in C$ (recall \cref{e:C}). 
And if $\alpha>0$,
then $1/\alpha>0$; thus, 
$x\in (1/\alpha) C$ by assumption and so $\alpha x \in C$.
In either case, we have shown that $\alpha x \in C$.
It now follows from \cref{f:rec1} that $x\in \rec(C)$. 

\cref{f:rec3}:
``$\Rightarrow$'':
We show the contrapositive and thus assume
that there is some 
$d\in X$ such that $C\cap \RP d$ is unbounded. 
Clearly, $d\neq 0$ and 
 there exists a sequence $(\rho_n)_\nnn$ in $\RPP$
such that $\rho_n\to\pinf$ and $c_n := \rho_n d \in C$. 
It follows that $d = (1/\rho_n)c_n$ and $1/\rho_n\to 0$. 
By \cref{f:rec0.5}, $d\in \rec(C)\smallsetminus\{0\}$. 
``$\Leftarrow$'':
Again, we show the contrapositive and thus assume that 
$d\in \rec(C)\smallsetminus\{0\}$.
By \cref{f:rec1}, $\RP d\subseteq C$ and so
$C\cap \RP d = \RP d$ is unbounded. 

\cref{f:rec4}: See \cite[Theorem~8.4]{Rock70}. 

\cref{f:rec5}: Clear from \cref{e:rec}. 
\end{proof}

It is clear that every linearly bounded set is bounded and 
that (by \cref{f:rec}\cref{f:rec3}\&\cref{f:rec4}) these notions
coincide provided that $X$ is finite-dimensional.
However, when $X$ is infinite-dimensional, then these notions do differ
as the next example illustrates\footnote{For more results
in this direction, see \cite{Phung}.}.

\begin{example}[\bf Goebel-Kuczumow]
(See \cite{GK1978}.)
Suppose that $X=\ell^2(I)$ where $I$ is infinite, 
let $(e_i)_{i\in I}$ be an orthonormal family in $X$
such that $\cspan\{e_i\}_{i\in I}=X$,
and let $(\beta_i)_{i\in I}$ be a family in $\RP$. 
Suppose that 
\begin{equation}
C = \menge{x = \sum_{i\in I}\xi_ie_i}{(\forall i\in I)\;\;|\xi_i|\leq \beta_i}. 
\end{equation}
Then $C$ is linearly bounded.
Moreover, $C$ is bounded $\Leftrightarrow$ $\sum_{i\in I}\beta_i^2 <\pinf$. 
In particular, if $X=\ell^2(\NN)$ and $C$ is given by 
\begin{equation}
C = \menge{x = (\xi_n)_\nnn\in\ell^2(\NN)}{(\forall \nnn)\;\;|\xi_n|\leq 1},
\end{equation}
then $C$ is unbounded but linearly bounded. 
\end{example}

The importance of the notion of the recession cone for our study
becomes apparent in the following results.

\begin{fact}
{(See \cite[Corollary~6.53]{BC2017}.)}
\label{f:rawsome}
Let $S$ be a nonempty closed convex subset of $(X\times\RR)\smallsetminus\{(0,0)\}$. 
Then
\begin{equation}
\clcone(S) = \cone(S)\uplus \rec(S).
\end{equation}
\end{fact}

\begin{corollary}
Recall that \cref{e:C} holds. 
Then 
\begin{subequations}
\begin{align}
\clcone\big(C\times\{1\} \big)
&=
\cone\big(C\times\{1\}\big)
\uplus 
\big(\rec(C)\times\{0\}\big),\label{e:rawsome+}\\
\clcone\big(C\times\{-1\} \big)
&=
\cone\big(C\times\{-1\}\big)
\uplus 
\big(\rec(C)\times\{0\}\big). \label{e:rawsome-}
\end{align}
\end{subequations}
\end{corollary}
\begin{proof}
Using \cref{f:rec}\ref{f:rec5}\&\ref{f:rec4}, 
we obtain  
$\rec(C\times\{1\})
= \rec(C)\times\rec(\{1\}) = \rec(C)\times\{0\}$ 
and 
$\rec(C\times\{-1\})
= \rec(C)\times\rec(\{-1\}) = \rec(C)\times\{0\}$. 
Because $(0,0)\notin C\times \{1\}$, 
the result now follows from \cref{f:rawsome}. 
\end{proof}

\begin{corollary}
Recall that \cref{e:C} holds, and
suppose that 
$C$ is linearly bounded (e.g., when $C$ is bounded). 
Then 
\begin{align}
\clcone\big(C\times\{1\} \big)
&=
\cone\big(C\times\{1\}\big)
\uplus 
\big\{(0,0)\big\} 
= 
\bigcup_{\rho \geq 0}\rho\big(C\times\{1\}\big). 
\end{align}
\end{corollary}
\begin{proof}
By \cref{f:rec}\cref{f:rec3}, $\rec(C) = \{0\}$; hence, 
the result follows from  \cref{e:rawsome+}. 
\end{proof}

\subsection{Polar set and polar cone}

We recall our standing assumption \cref{e:C}.
We denote by 
\begin{equation}
\sigma_C := x\mapsto \sup_{c\in C}\scal{c}{x},
\quad
\barc(C) := \dom \sigma_C = \menge{x\in X}{\sigma_C(x)<\pinf},
\end{equation}
\begin{equation}
C^\odot := \menge{x\in X}{\sigma_C(x)\leq 1}
= \menge{x\in X}{(\forall c\in C)\;\;\scal{c}{x}\leq 1}, 
\end{equation}
\begin{equation}
C^\ominus := \menge{x\in X}{\sigma_C(x)\leq 0}
= \menge{x\in X}{(\forall c\in C)\;\;\scal{c}{x}\leq 0}, 
\end{equation}
the \emph{support function},
the \emph{barrier cone}, 
the \emph{polar set},
and the 
\emph{polar cone} of $C$, 
respectively.

Here are some results concerning these objects. 
\begin{fact}
\label{p:lastlec}
The following hold.
\begin{enumerate}
\item 
\label{p:lastlec0}
$\sigma_C$ is sublinear, lower semicontinuous, and proper. 
\item 
\label{p:lastlec1}
$\barc(C) = \cone(C^\odot)$.
\item 
\label{p:lastlec2}
$(\barc(C))^\ominus = \rec(C)$.
\item 
\label{p:lastlec5}
$0\in C^\odot$, and $C^\odot$ is convex and closed. 
\item 
\label{p:lastlec6}
$C^{\odot\odot} = C$. 
\item 
\label{p:lastlec8}
$\barc(C^\odot) = \cone(C)$. 
\item 
\label{p:lastlec8.5}
$C^\ominus = (\cone(C))^\ominus = (\clcone(C))^\ominus$. 
\item 
\label{p:lastlec9}
$(\barc(C^\odot))^\ominus = \rec(C^\odot)$. 
\item 
\label{p:lastlec10}
$C^\ominus = \rec(C^\odot)$. 
\end{enumerate}
\end{fact}
\begin{proof}
Let $x\in X$. 

\cref{p:lastlec0}: See \cite[Example~11.2]{BC2017}.

\cref{p:lastlec1}:
``$\subseteq$'': 
Suppose that $x\in\barc(C)$.
Then $\sigma_C(x)<\pinf$.
If $\sigma_C(x)\leq 1$, then $x\in C^\odot \subseteq \cone(C^\odot)$.
And if $\sigma_C(x)>1$, then $0<1/\sigma_C(x)<1$ and 
$1/\sigma_C(x)x\in C^\odot$; consequently,
$x\in\sigma_C(x)C^\odot \subseteq \cone(C^\odot)$. 
``$\supseteq$'': 
Suppose that $x\in\cone(C^\odot)$.
Then there exists $\alpha>0$ such that $x \in \alpha C^\odot$. 
Hence $x/\alpha \in C^\odot$, i.e., $\sigma_C(x/\alpha)\leq 1$ and
hence $\sigma_C(x)\leq\alpha<\pinf$. 

\cref{p:lastlec2}:
See \cite[Proposition~6.49(v)]{BC2017}. 

\cref{p:lastlec5}: 
This is clear from the definition of the polar set. 

\cref{p:lastlec6}:
Combine \cref{e:C} with \cite[Corollary~7.19(i)]{BC2017}. 

\cref{p:lastlec8}:
Using \cref{p:lastlec5}, \cref{p:lastlec1}, and \cref{p:lastlec6}, we have 
$\barc(C^\odot) = \cone(C^{\odot\odot})=\cone(C)$. 

\cref{p:lastlec8.5}:
This follows readily from the definitions. 

\cref{p:lastlec9}:
By \cref{p:lastlec5}, $C^\odot$ is convex, closed, and contains $0$.
Now apply  \cref{p:lastlec2} to $C^\odot$. 

\cref{p:lastlec10}: 
Using \cref{p:lastlec8.5}, \cref{p:lastlec8}, and \cref{p:lastlec9}, we deduce that 
$C^\ominus = (\cone(C))^\ominus = (\barc(C^\odot))^\ominus = \rec(C^\odot)$. 
\end{proof}

\subsection{Examples of polar pairs}

In this subsection, we will present various examples of polar pairs.
These are known in the literature; however, no proofs are presented.
For completeness, we include proofs here and in \cref{app:A}. 

\begin{example}{\bf (balls centred at $0$)}
\label{ex:duno}
Let $\|\cdot\|$ be an arbitrary norm, let $\gamma>0$, and suppose that
\begin{equation}
C = \menge{x\in X}{\|x\|\leq \gamma}. 
\end{equation}
Then 
\begin{equation}
C^\odot = \menge{y\in X}{\|y\|_*\leq 1/\gamma},
\end{equation}
where $\|\cdot\|_*$ denotes the dual norm. 
\end{example}
\begin{proof}
Let $y\in X$. We have, using the definition of the dual norm, 
\begin{equation}
\sigma_C(y) = \sup_{\|c\|\leq 1}\scal{y}{\gamma c}=\gamma\|y\|_*; 
\end{equation}
consequently, 
$y\in C^\odot$
$\Leftrightarrow$
$\gamma\|y\|_*\leq 1$
$\Leftrightarrow$
$\|y\|_*\leq 1/\gamma$ and the result follows.
\end{proof}

\begin{example}
\label{ex:viola}
Suppose that 
\begin{equation}
C = \overline{D + R},
\end{equation}
where $D$ is a closed convex subset of $X$, $0\in D$, 
and $R$ is a nonempty closed convex cone in $X$.
Then 
\begin{equation}
C^\odot = D^\odot \cap R^\ominus. 
\end{equation}
\end{example}
\begin{proof}
Indeed, $\sigma_C = \sigma_D + \sigma_R = \sigma_D + \iota_{R^\ominus}$ 
from which the result follows. 
\end{proof}

\begin{example}
\label{ex:shifty}
Suppose that 
\begin{equation}
C = D-d, 
\end{equation}
where $D$ is a closed convex subset of $X$ and $d\in D$.
Then 
$\sigma_C\colon y\mapsto   \sigma_D(y) - \scal{d}{y}$ and hence
\begin{equation}
C^\odot = \menge{y\in X}{\sigma_D(y)\leq 1+\scal{d}{y}}.
\end{equation}
\end{example}
\begin{proof}
We have $\sigma_C = \sigma_D - \scal{d}{\cdot}$ 
from which the result follows. 
\end{proof}

\begin{example}
\label{ex:shiftyball}
Let $D$ be the closed unit ball centered at $0$ in $X$, and
suppose that $d\in D$ such that $\|d\|=1$ and 
\begin{equation}
C = D - d.
\end{equation}
Then $\sigma_C\colon y\mapsto \|y\|-\scal{d}{y}$ and  
\begin{equation}
C^\odot = \menge{y\in X}{\|P_E(y)\|^2\leq 1+ 2\scal{d}{y}},
\end{equation}
where $E = \{d\}^\perp$. 
\end{example}
\begin{proof}
Let $y\in X$.
Using \cref{ex:shifty}, we see that 
\begin{equation}
\sigma_C(y)= \sigma_D(y)-\scal{d}{y}=\|y\|-\scal{d}{y}. 
\end{equation}
Set $E = \{d\}^\perp$. 
Note that $P_{\RR d}(y)=\scal{d}{y}d$ and thus 
$\|y\|^2-\|P_E(y)\|^2 = \|P_{\RR d}(y)\|^2 = \scal{d}{y}^2$. 
We therefore have the equivalences
\begin{subequations}
\begin{align}
y\in C^\odot
&\Leftrightarrow
\sigma_C(y)\leq 1\\
&\Leftrightarrow
\|y\|-\scal{d}{y}\leq 1\\
&\Leftrightarrow
\|y\|\leq 1+\scal{d}{y}\\
&\Leftrightarrow
\|y\|^2\leq \big(1+\scal{d}{y}\negthinspace\big)^2 \;\land\; 0\leq 1+\scal{d}{y}\\
&\Leftrightarrow
\|y\|^2\leq 1+\scal{d}{y}^2+2\scal{d}{y} \;\land\; 0\leq 1+\scal{d}{y}\\
&\Leftrightarrow
\|y\|^2-\scal{d}{y}^2 \leq 1+2\scal{d}{y} \;\land\; 0\leq 1+\scal{d}{y}\\
&\Leftrightarrow
\|P_E(y)\|^2\leq 1+2\scal{d}{y} \;\land\; 0\leq 1+\scal{d}{y}\\
&\Leftrightarrow
\|P_E(y)\|^2\leq 1+2\scal{d}{y} \;\land\; 0 \leq 1+2\scal{d}{y}\;\land\;0\leq 1+\scal{d}{y}\\
&\Leftrightarrow
\|P_E(y)\|^2\leq 1+2\scal{d}{y} \;\land\; -1/2 \leq \scal{d}{y}\;\land\;0\leq 1+\scal{d}{y}\\
&\Leftrightarrow
\|P_E(y)\|^2\leq 1+2\scal{d}{y} \;\land\; 1/2 \leq 1+\scal{d}{y}\;\land\;0\leq 1+\scal{d}{y}\\
&\Leftrightarrow
\|P_E(y)\|^2\leq 1+2\scal{d}{y} \;\land\; 1/2 \leq 1+\scal{d}{y}\\
&\Leftrightarrow
\|P_E(y)\|^2\leq 1+2\scal{d}{y}
\end{align}
\end{subequations}
and this completes the proof. 
\end{proof}

As a consequence of \cref{ex:duno} and the fact that the Euclidean norm is self-dual, we 
obtain the following: 
\begin{example}
\label{ex:1}
Let $\gamma>0$ and suppose that 
\begin{equation}
C = \menge{x\in X}{\|x\|\leq \gamma}
\end{equation}
Then
\begin{equation}
C^\odot = \menge{y\in X}{\|y\|\leq 1/\gamma}.
\end{equation}
\end{example}

The following four examples are from \cite[page~126]{Rock70}. 

\begin{example}
Suppose that $X=\RR^n$ and 
\begin{equation}
C = \menge{(x_1,\ldots,x_n)\in \RR^n}{x_1+\cdots+x_n\leq 1,\;\text{each } x_i\geq 0}.
\end{equation}
Then
\begin{equation}
C^\odot = \menge{(y_1,\ldots,y_n)\in \RR^n}{\text{each } y_i\leq 1}.
\end{equation}
\end{example}
\begin{proof}
See \cref{app:1}. 
\end{proof}

\begin{example}
\label{ex:1b}
(See also \cref{fig41}.) 
Suppose that $X=\RR^n$ and 
\begin{equation}
C = \menge{(x_1,\ldots,x_n)\in \RR^n}{|x_1|+\cdots+|x_n|\leq 1}.
\end{equation}
Then
\begin{equation}
C^\odot = \menge{(y_1,\ldots,y_n)\in \RR^n}{\text{each } |y_i|\leq 1}.
\end{equation}
\end{example}
\begin{proof}
Clear from \cref{ex:duno} because $\|\cdot\|_1^* = \|\cdot\|_\infty$. 
\end{proof}

\begin{figure}[ht]
\centering
\includegraphics[width=0.5\textwidth]{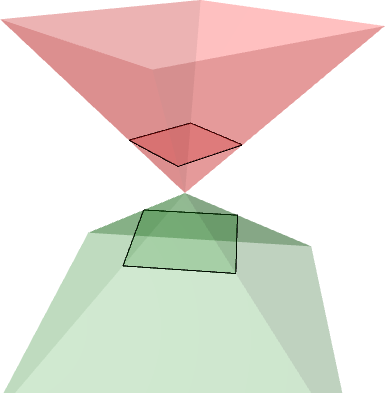}
\caption{$\clcone(C\times\{1\})$ and its polar cone, using $C$ from \cref{ex:1b}, with $n = 2$}
\label{fig41}
\end{figure}

\begin{example}
Suppose that $X=\RR^2$ and 
\begin{equation}
C = \menge{(x_1,x_2)\in \RR^2}{(x_1-1)^2+x_2^2\leq 1}.
\end{equation}
Then
\begin{equation}
C^\odot = \menge{(y_1,y_2)\in \RR^2}{y_1\leq (1-y_2^2)/2}.
\end{equation}
\end{example}
\begin{proof}
Note that $C=D-d$, where $D=$ closed unit ball centered at $(0,0)$ and $d=(-1,0)$. 
Then $E = \{d\}^\perp = \RR(0,1)$.
Let $y=(y_1,y_2)\in\RR^2$. 
Then \cref{ex:shiftyball} yields
$\sigma_C(y)=\|y\|-\scal{d}{y}=\tsqrt{y_1^2+y_2^2}+y_1$ and 
\begin{subequations}
\begin{align}
y\in C^\odot
&\Leftrightarrow
\|P_E(y)\|^2 \leq 1+2\scal{d}{y}\\
&\Leftrightarrow
y_2^2\leq 1-2y_1\\
&\Leftrightarrow
y_1 \leq (1-y_2^2)/2,
\end{align}
\end{subequations}
and we are done!
\end{proof}

\begin{example}
\label{ex:hyp}
(See also \cref{fig31}.)
Suppose that $X=\RR^2$ and 
\begin{equation}
C = \Menge{(x_1,x_2)\in \RR^2}{x_1\leq 1 - \sqrt{1+x_2^2}}.
\end{equation}
Then
\begin{equation}
C^\odot = \conv \Big(\{0\}\cup \Menge{(y_1,y_2)\in \RR^2}{y_1\geq (1+y_2^2)/2}\Big).
\end{equation}
\end{example}
\begin{proof}
See \cref{app:3}, where we not only provide a simpler description of $C^\odot$ but also
a formula for $\sigma_C$. 
\end{proof}

\begin{figure}
\centering
	\includegraphics[width=0.5\textwidth]{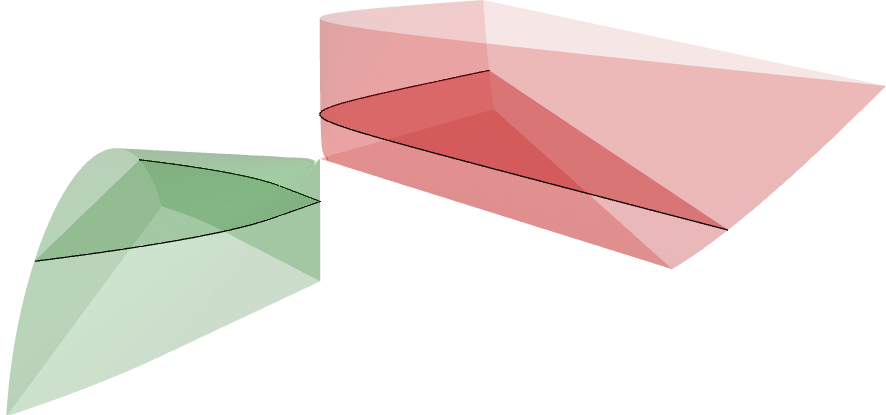}
	\caption{$\clcone(C\times\{1\})$ and its polar cone, using $C$ from \cref{ex:hyp}}
\label{fig31}
\end{figure}

Here are some examples from \cite[page~136]{Rock70}.

\begin{example}
Suppose that $X=\RR^n$, let $1<p<\pinf$, set 
\begin{equation}
C = \Menge{(x_1,\ldots,x_n)\in \RR^n}{|x_1|^p+\cdots+|x_n|^p\leq 1}
\end{equation}
and suppose $1<q<\pinf$ satisfies $\frac{1}{p}+\frac{1}{q}=1$. 
Then
\begin{equation}
C^\odot = \Menge{(y_1,\ldots,y_n)\in \RR^n}{|y_1|^q+\cdots+|y_n|^q\leq 1}.
\end{equation}
\end{example}
\begin{proof}
Clear from \cref{ex:duno} because the dual norm of the $p$-norm is the $q$-norm. 
\end{proof}

\begin{example}
Suppose that $X=\RR^n$, let $Q\in\RR^{n\times n}$ be symmetric and positive
definite, and set 
\begin{equation}
C = \Menge{x\in \RR^n}{\scal{x}{Qx}\leq 1}
\end{equation}
Then
\begin{equation}
C^\odot = \Menge{y\in \RR^n}{\tscal{y}{Q^{-1}y}\leq 1}.
\end{equation}
\end{example}
\begin{proof}
This follows from \cref{ex:duno} because the dual norm of the $Q$-norm is the $Q^{-1}$-norm. 
\end{proof}

\begin{example}
\label{ex:cbp}
(See also \cref{fig22})
Suppose that $X = \RR^2$, and set
\begin{equation}
C = \Menge{(x,y) \in \RR^2}{x^2 + y^2 \le 1 \text{ or } (|x| \le 1 \text{ and } y \ge 0)}
\end{equation}
Then
\begin{equation}
C^\odot = \Menge{(x,y) \in \RR^2}{x^2 + y^2 \le 1 \text{ and } y \le 0)}.
\end{equation}
\end{example}
\begin{proof}
This follows from \cref{ex:viola} with $D=D^\odot=$ the Euclidean unit ball 
(see \cref{ex:1}) and 
$R = \{0\}\times\RR_+$ for which $R^\ominus = \RR\times\RM$. 
\end{proof}

\begin{figure}
\centering
\includegraphics[width=0.5\textwidth]{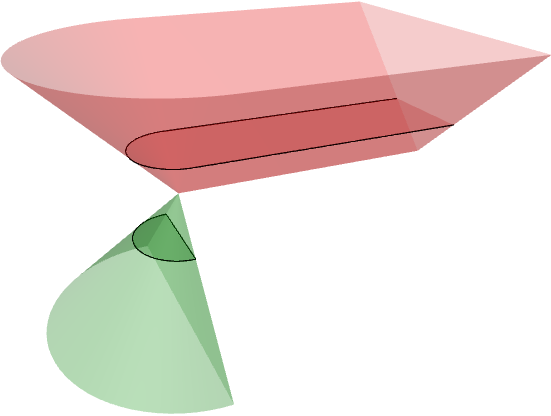}
\caption{$\clcone(C\times\{1\})$ and its polar cone, using $C$ from \cref{ex:cbp}}
\label{fig22}
\end{figure}

\section{Polar cone of the homogenization cone}

\label{sec:polar}

\begin{theorem}{\bf (polar cone)}
Recall that \cref{e:C} holds. 
Then
\begin{subequations}
\begin{align}
\big(\clcone(C\times\{1\}\big)^\ominus
&= \cone(C^\odot\times\{-1\})\uplus \big(C^\ominus \times\{0\}\big)
\label{e:pc1}\\
&= \cone(C^\odot\times\{-1\})\uplus \big(\rec(C^\odot) \times\{0\}\big)
\label{e:pc2}\\
&= \clcone\big(C^\odot\times\{-1\}).
\label{e:pc3}
\end{align}
\end{subequations}
\end{theorem}
\begin{proof}
We first prove \cref{e:pc1}. 
Let $d\in X$. 
Then 
\begin{subequations}
\begin{align}
(d,-1)\in \big(\clcone(C\times\{1\})\big)^\ominus
&\Leftrightarrow
(d,-1)\in (C\times\{1\})^\ominus \\
&\Leftrightarrow
(\forall c\in C)\;\; \scal{(d,-1)}{(c,1)}\leq 0\\
&\Leftrightarrow
(\forall c\in C)\;\; 
\scal{d}{c}-1\leq 0\\
&\Leftrightarrow d\in C^\odot.
\end{align}
\end{subequations}
This shows $C^\odot\times\{-1\}\subseteq (\clcone(C\times\{1\}))^\ominus$.
Because $(\clcone(C\times\{1\}))^\ominus$ is a closed cone, we deduce that 
\begin{equation}
\label{e:211128a}
\cone\big(C^\odot\times\{-1\}\big) \subseteq 
\clcone\big(C^\odot\times\{-1\}\big)\subseteq \big(\clcone(C\times\{1\})\big)^\ominus.
\end{equation}
If $d\in C^\ominus$, then 
$(\forall c\in C)$ 
$\scal{(d,0)}{(c,1)} = \scal{d}{c}\leq 0$; 
hence $(d,0)\in (C\times\{1\})^\ominus$.
Thus $C^\ominus \times\{0\}\subseteq (C\times\{1\})^\ominus$ and therefore
\begin{equation}
\label{e:211128b}
C^\ominus \times\{0\}\subseteq \big(\clcone(C\times\{1\})\big)^\ominus.
\end{equation}
Combining \cref{e:211128a} and \cref{e:211128b} yields 
\begin{equation}
\label{e:211128c}
\cone\big(C^\odot\times\{-1\}\big) \uplus 
\big(C^\ominus \times\{0\}\big)\subseteq \big(\clcone(C\times\{1\})\big)^\ominus.
\end{equation}
Conversely, let 
\begin{equation}
\label{e:211128d}
(d,s)\in \big(\clcone(C\times\{1\})\big)^\ominus.
\end{equation}
Then 
\begin{equation}
\label{e:211128e} 
(\forall c\in C)\quad
\scal{d}{c}+s=\scal{(d,s)}{(c,1)}\leq 0.
\end{equation}
Choosing $c=0$ gives 
\begin{equation}
s\leq 0.
\end{equation}
We distinguish two cases.

\emph{Case~1:}
$s=0$.\\
Then \cref{e:211128e} gives $(\forall c\in C)$ 
$\scal{d}{c}\leq 0$, i.e., $d\in C^\ominus$ and thus 
\begin{equation}
\label{e:211128g}
(d,s)\in C^\ominus\times\{0\}. 
\end{equation}

\emph{Case~2:}
$s=-|s|<0$.\\
Dividing \cref{e:211128e} by $|s|>0$ gives $(\forall c\in C)$ 
$\scal{d/|s|}{c} - 1 \leq 0$.
Hence $d/|s|\in C^\odot$, which implies 
$(d/|s|,-1)\in C^\odot\times\{-1\}$ and further
\begin{equation}
\label{e:211128i}
(d,s)=|s|\big(d/|s|,-1\big)\in |s|\big(C^\odot\times\{-1\}\big) \subseteq 
\cone\big(C^\odot\times\{-1\}\big). 
\end{equation}
Combining \cref{e:211128d}, \cref{e:211128g}, and \cref{e:211128i}, 
we see that 
\begin{equation}
\label{e:211128j}
\big(\clcone(C\times\{1\})\big)^\ominus
\subseteq
\cone\big(C^\odot\times\{-1\}\big)
\uplus
\big(C^\ominus\times\{0\}\big). 
\end{equation}

In turn, combining \cref{e:211128c} with \cref{e:211128j} yields \cref{e:pc1}. 

Equation~\cref{e:pc2} follows immediately from 
\cref{p:lastlec}\cref{p:lastlec10}. 

Finally, apply \cref{e:rawsome-} to the set $C^\odot$ to obtain \cref{e:pc3}. 
\end{proof}

\section{Scaled sets: projection and (squared) distance function}

\label{sec:scaled}

Throughout this section, we fix 
\begin{empheq}[box=\mybluebox]{equation}
\label{e:y}
y\in X.
\end{empheq}
Recall our standing assumption \cref{e:C}. 
It is well known (see, e.g., \cite[Proposition~29.1(ii)]{BC2017}) 
that the projection operator satisfies
\begin{equation}
\label{e:PaC}
(\forall \alpha>0)\quad 
P_{\alpha C}(y) = \alpha P_C\big({y}/{\alpha}\big);
\end{equation}
in turn, we have 
\begin{equation}
\label{e:daC}
(\forall \alpha>0)\quad 
d_{\alpha C}(y) = \alpha d_C(y/\alpha). 
\end{equation}
The limiting behaviour of \cref{e:PaC} and \cref{e:daC} is understood:
Combining \cref{e:PaC} and \cref{f:rec}\ref{f:rec2} with 
\cite[Proposition~29.8 and Proposition~29.7]{BC2017} yields 
\begin{equation}
\label{e:Prec}
\lim_{\alpha\to 0^+}
\alpha P_C\big(\tfrac{y}{\alpha}\big)
= \lim_{\alpha\to 0^+} P_{\alpha C}(y) = P_{\rec(C)}(y)
\end{equation}
and 
\begin{equation}
\label{e:Precinf}
\lim_{\alpha\to \pinf}
\alpha P_C\big(\tfrac{y}{\alpha}\big)
= \lim_{\alpha\to \pinf} P_{\alpha C}(y) = P_{\clcone(C)}(y). 
\end{equation}
The limit statements 
\cref{e:Prec} and \cref{e:Precinf} yield 
\begin{equation}
\label{e:drec}
\lim_{\alpha\to 0^+} \alpha d_{C}(y/\alpha) 
= \lim_{\alpha\to 0^+} d_{\alpha C}(y) 
= d_{\rec(C)}(y)
\end{equation}
and 
\begin{equation}
\label{e:drecinf}
\lim_{\alpha\to \pinf} \alpha d_{C}(y/\alpha) 
= \lim_{\alpha\to \pinf} d_{\alpha C}(y) 
= d_{\clcone(C)}(y),
\end{equation}
respectively.

We now define the function
\begin{empheq}[box=\mybluebox]{equation}
\label{e:phi}
\phi\colon \RPP\to\RR\colon \alpha\mapsto d^2_{\alpha C}(y)=\alpha^2 d^2_C(y/\alpha),
\end{empheq}
where we used \cref{e:daC} for the identity. 

\begin{theorem}[\bf convexity of $\phi$]
\label{t:phi}
The function $\RPP\to\RR\colon \alpha \mapsto d_{\alpha C}(y)$ is convex.
Moreover, the function $\phi$ defined in \cref{e:phi} is  convex, 
differentiable with 
\begin{equation}
\label{e:211207a}
\phi'(\alpha) 
=-2\alpha\scal{P_C(y/\alpha)}{(y/\alpha)-P_C(y/\alpha)}\leq 0,
\end{equation}
decreasing, and bounded below by $0$. 
In fact, 
\begin{equation}
\label{e:211207b}
\lim_{\alpha\to 0^+}\phi(\alpha) = \sup_{\alpha>0} \phi(\alpha) = d^2_{\rec(C)}(y)
\end{equation}
and 
\begin{equation}
\label{e:211207c}
\lim_{\alpha\to \pinf}\phi(\alpha) = \inf_{\alpha>0} \phi(\alpha) = d^2_{\clcone(C)}(y).
\end{equation}
\end{theorem}
\begin{proof}
The function $\beta \mapsto d_C(\beta y)$ is convex;
therefore, its \emph{adjoint} $\alpha \mapsto \alpha d_C(y/\alpha)=d_{\alpha C}(y)$ 
is convex 
on $\RPP$ by, e.g., \cite[Example~8.33]{BC2017} and \cref{e:daC}. 
The convexity of $\phi$ now follows from 
the composition rule \cite[Proposition~8.21]{BC2017}. 

We now turn to differentiability of $\phi$.
Using the fact that $\nabla d^2_C = 2(\Id-P_C)$ 
(see, e.g., \cite[Corollary~12.31]{BC2017}) and the chain rule, 
we compute 
\begin{subequations}
\begin{align}
\phi'(\alpha)
&= \frac{d}{d\alpha} \big(\alpha^2 d_C^2(y/\alpha)\big)\\
&=2\alpha d_C^2(y/\alpha) + \alpha^2 2\bscal{(\Id-P_C)(y/\alpha)}{-y/\alpha^2}\\
&=2\alpha\big(\|y/\alpha\|^2 + \|P_C(y/\alpha)\|^2 - 2\scal{y/\alpha}{P_C(y/\alpha)}\big)\\
&\qquad 
-2\|y\|^2/\alpha+2\scal{P_C(y/\alpha)}{y}\\
&= 2\alpha \|P_C(y/\alpha)\|^2 -2\scal{y}{P_C(y/\alpha)}\\
&=2\alpha\scal{P_C(y/\alpha)}{P_C(y/\alpha)-(y/\alpha)}\\
&=-2\alpha\scal{P_C(y/\alpha)}{(\Id-P_C)(y/\alpha)},
\end{align}
\end{subequations}
as claimed.
Because $P_C(0)=0$ (by \cref{e:C}) and $P_C$ is firmly nonexpansive 
(see \cite[Proposition~4.16]{BC2017}), 
we obtain from \cite[Proposition~4.4]{BC2017} that 
$\scal{P_C(y/\alpha)}{(y/\alpha)-P_C(y/\alpha)}
=
\scal{P_C(y/\alpha)-P_C(0)}{(\Id-P_C)(y/\alpha)-(\Id-P_C)(0)}\geq 0$ 
which yields the inequality in \cref{e:211207a}. 
Therefore, $\phi$ is decreasing and clearly bounded below by $0$. 
Combining this with \cref{e:drec} and \cref{e:drecinf}, we obtain 
\cref{e:211207b} and \cref{e:211207c}. 
\end{proof}

We have now all the tools together to study the 
lower semicontinuous hull of $\phi$ which we also define in the next result. 

\begin{theorem}[\bf lower semicontinuous hull of $\phi$]
\label{t:Phi}
Recall the definition of $\phi$ in \cref{e:phi}. 
Then the function  
\begin{empheq}[box=\mybluebox]{equation}
\label{e:Phi}
\Phi\colon \RR\to\RR\colon \alpha \mapsto 
\begin{cases}
\pinf, &\text{if $\alpha<0$;}\\
d^2_{\rec(C)}(y), &\text{if $\alpha=0$;}\\
\phi(\alpha) = d^2_{\alpha C}(y), &\text{if $\alpha>0$}
\end{cases}
\end{empheq}
is the lower semicontinuous hull of $\phi$. 
Moreover, 
\begin{equation}
\label{e:Phi'}
(\forall \alpha>0)\quad
\Phi'(\alpha)
= 
-2\alpha\scal{P_C(y/\alpha)}{(y/\alpha)-P_C(y/\alpha)}\leq 0,
\end{equation}
\begin{equation}
\label{e:Phi'(0)}
\minf = \Phi'_-(0)
\leq \Phi'_+(0)=\lim_{\alpha\to 0^+}\Phi'(\alpha) = \inf_{\alpha>0}\Phi'(\alpha) \leq 0, 
\end{equation}
and 
\begin{equation}
\label{e:Phi'(pinf)}
\lim_{\alpha\to\pinf}\Phi'(\alpha) = 0.
\end{equation}
\end{theorem}
\begin{proof}
The fact that $\Phi$ is the lower semicontinuous hull of $\phi$ 
follows from \cite[Proposition~9.33]{BC2017} and \cref{e:211207b}. 
The formula for $\Phi'(\alpha)$ presented in \cref{e:Phi'} is clear from 
\cref{e:211207a}, while 
the formula for $\Phi'_{\mp}(0)$ in \cref{e:Phi'(0)} 
is a consequence of \cite[Theorem~24.1]{Rock70}. 
Because $\phi$ is decreasing (see \cref{t:phi}), so is $\Phi$. 
Hence $\sigma := \sup\Phi'(\RPP)=\lim_{\alpha\to\pinf}\Phi'(\alpha) \leq 0$.
On the other hand, because $\Phi$ is bounded below, 
we deduce from \cite[Corollary~16.52]{BC2017} that $0\in\cran\Phi'$.
Altogether, the proof is complete.
\end{proof}

\begin{remark}
Suppose that $C$ is bounded.
Then the set $\{P_C(y/\alpha)\}_{\alpha>0}$ is bounded, 
and $\lim_{\alpha\to 0^+} (y-\alpha P_C(y/\alpha)) = y 
- P_{\rec(C)}(y) = y$ using \cref{e:Prec} and 
\cref{f:rec}\cref{f:rec3}. 
In view of \cref{e:Phi'} and \cref{e:Phi'(0)}, we deduce that 
\begin{equation}
\minf<\Phi'_+(0) \leq 0. 
\end{equation}
\end{remark}

\begin{remark}
Suppose that $C$ is a cone. 
Then $(\forall\alpha>0)$ $C = \alpha C = \rec(C)$, 
and hence $\Phi = d^2_C(y) + \iota_{\RP}$. 
Thus $\Phi'_+(0) = 0$. 
\end{remark}

\section{Projection onto the homogenization cone}

\label{sec:proj}

Throughout this section, we fix 
\begin{empheq}[box=\mybluebox]{equation}
\label{e:ys}
(y,s) \in X\times\RR.
\end{empheq}
Recalling the definition \cref{e:K}, our goal is to determine
$P_K(y,s)$. 
From \cref{e:rawsome+} it is clear that 
\begin{equation}
\text{either~}
P_K(y,s) \in \cone(C\times\{1\})
\text{~or~}
P_K(y,s)\in \rec(C)\times\{0\}
\end{equation}
but not both\footnote{We always use ``either/or'' in the mutually exclusive sense.}.
Note that $P_K(y,s)$ is either $P_{\cone(C\times\{1\})}(y,s)$ or 
$P_{\rec(C)\times\{0\}}(y,s)$; however, 
note that the vector $P_{\cone(C\times\{1\})}(y,s)$ need not exist because 
$\cone(C\times\{1\})$ is convex but not necessarily closed.
On the other hand,
$P_{\rec(C)\times\{0\}}(y,s)$ always exists and is equal to 
$(P_{\rec(C)}(y),0)$.
Altogether, if both vectors exist, then $P_K(y,s)$ is the one that is
closer to $(y,s)$.
This discussion leads us to investigate $P_{\cone(C\times\{1\})}(y,s)$.
To this end,  we pick a generic element from 
$\cone(C\times \{1\})$, say 
\begin{equation}
\alpha(c,1),
\quad
\text{where $\alpha>0$ and $c\in C$.}
\end{equation}
Then the squared distance between $\alpha(c,1)=(\alpha c,\alpha)$ and $(y,s)$ 
is given by 
\begin{align}
\label{e:211130a}
\|(\alpha c,\alpha)-(y,s)\|^2
&= \|\alpha c-y\|^2 + (\alpha-s)^2. 
\end{align}
For the next little while, we think of $\alpha>0$ being fixed.
Minimizing \cref{e:211130a} over $c\in C$ 
leads us to the optimal $c$ being 
$\tfrac{1}{\alpha}P_{\alpha C}(y)= P_C(y/\alpha)$ 
using \cref{e:PaC}.
Having found the optimal $c$ for a fixed $\alpha>0$, we now
vary $\alpha$ over the positive reals.
In other words, we are interested in minimizing the function
\begin{subequations}
\label{e:psi}
\begin{empheq}[box=\mybluebox]{align}
\psi \colon \RPP\to\RR \colon \alpha \mapsto 
& \|P_{\alpha C}(y)-y\|^2 + (\alpha-s)^2
= d^2_{\alpha C}(y)+(\alpha - s)^2\\
&= \alpha^2\|(y/\alpha)-P_C(y/\alpha)\|^2 + (\alpha - s)^2,
\end{empheq}
\end{subequations}
where we used \cref{e:PaC} and \cref{e:daC}. 
Note that $\psi$ is the sum of $\phi$ from \cref{e:phi} and a quadratic.
Hence, our work in \cref{sec:scaled} pays off and leads readily to the following result.

\begin{theorem}[\bf $P_K(y,s)$ and the lower semicontinuous hull of $\psi$]
\label{t:Psi}
Recall the definition of $\psi$ in \cref{e:psi}.
Then the function 
\begin{empheq}[box=\mybluebox]{equation}
\label{e:Psi}
\Psi\colon\RR\to\RR\colon \alpha \mapsto 
\begin{cases}
\pinf, &\text{if $\alpha<0$;}\\
d^2_{\rec C}(y)+s^2, &\text{if $\alpha=0$;}\\
\alpha^2d^2_{C}(y/\alpha)+(\alpha-s)^2, &\text{if $\alpha>0$.}
\end{cases}
\end{empheq}
is the lower semicontinuous hull of $\psi$. 
The function $\Psi$ is strictly convex and supercoercive --- 
we denote its unique minimizer by $\alpha^*\geq 0$.
We then have 
\begin{equation}
\label{e:P_K}
P_K(y,s) = 
\begin{cases}
\big(P_{\rec(C)}(y),0\big), &\text{if $\alpha^*=0$;}\\
\big(\alpha^*P_{C}(y/\alpha^*),\alpha^*\big), &\text{if $\alpha^*>0$.}
\end{cases}
\end{equation}
\end{theorem}
\begin{proof}
The formula for \cref{e:Psi} is clear from \cref{e:Phi} and the definition 
of $\psi$. 
The quadratic term $\alpha\mapsto (\alpha-s)^2$ ensures that 
$\Psi$ is strictly convex and coercive;
in turn, this guarantees the existence and uniqueness of $\alpha^*$. 
The formula \cref{e:P_K} for $P_K(y,s)$ is clear by the discussion
preceding this theorem.
\end{proof}

Because we assume that $P_C$ is available, it is easy to evaluate
$\Psi$ at positive real numbers; however, $0$ is challenging 
because we don't assume that we have access to the recession cone and 
its projection.
Our ultimate goal is to present an algorithm for finding the minimizer of $\Psi$;
thus, we collect derivative information next.

\begin{proposition}
The function $\Psi$ defined in \cref{e:Psi} is differentiable on $\RPP$:
if $\alpha>0$, then
\begin{equation}
\Psi'(\alpha)
= 2(\alpha-s) - 2\alpha\scal{P_C(y/\alpha)}{(y/\alpha)-P_C(y/\alpha)}.
\end{equation}
Moreover, 
\begin{equation}
\Psi'_+(0) = -2s -2\lim_{\alpha\to 0^+}\alpha\scal{P_C(y/\alpha)}{(y/\alpha)-P_C(y/\alpha)} 
\end{equation}
and 
\begin{equation}
\lim_{\alpha\to\pinf}\Psi'(\alpha) = \pinf.
\end{equation}
\end{proposition}
\begin{proof}
The function $\Psi$ is the sum of $\Phi$ from \cref{e:Phi} 
and $\alpha\mapsto (\alpha-s)^2$.
Hence the result follows from \cref{t:Phi}.
\end{proof}

The function $\Psi$ is a supercoercive strictly convex function with domain 
$\RP$ such that $\lim_{\alpha\to\pinf}\Psi(\alpha)=\pinf$.
Denote the unique minimizer of $\Psi$ by $\alpha^*$. 
Consider $\Psi_+'(0)$.
If $\Psi_+'(0)\geq 0$, then $\Psi$ must be increasing and thus 
$\alpha^*=0$.
If $\Psi_+'(0)<0$, then $\Psi$ first decreases before it again
increases and so $\alpha^*>0$.
Unfortunately, we don't have a simple formula for $\Psi'_+(0)$.

Nonetheless, the above discussion leads to the following algorithm for finding $\alpha^*$ and hence an approximation of $P_K(y,s)$: 

\begin{algo}[\bf finding $P_K$]
\label{algo1}
Given are $0<\alpha<\beta$ (say $\alpha=1$ and $\beta=2$) and
an error tolerance $\epsilon>0$ (say $\epsilon = 10^{-6}$), and $n=0$. 
Now proceed as follows.
{\rm 
\begin{itemize}
\item[1.] If $\beta-\alpha<\varepsilon$, 
then STOP with $\alpha^*=\alpha$. 
\item[2.] If $\Psi'(\alpha)=0$, then STOP with $\alpha^*=\alpha$; 
if $\Psi'(\beta)=0$, then STOP with $\alpha^*=\beta$. 
\item[3.] If $\Psi'(\alpha)<0<\Psi'(\beta)$, 
then perform a standard bisection to find $\alpha^*$ 
with $\Psi'(\alpha^*)=0$ in the interval 
$[\alpha,\beta]$. (This is justified because $\Psi'$ is continuous on $\RPP$.)
\item[4.] If $0<\Psi'(\alpha)$, then update 
$\beta := \alpha$, $\alpha := \alpha/2$, $n := n+1$, and return to STEP~1. 
\item[5.] If $\Psi'(\beta)<0$, then update
$\alpha := \beta$, $\beta:= 2\beta$, $n := n+1$, and return to STEP~1. 
\end{itemize}
} 
Then 
\begin{equation}
\big(\alpha^*P_C(y/\alpha^*),\alpha^*\big)
\end{equation}
serves as an approximation of $P_K(y,s)$. 
\end{algo}

\section{Examples}
\label{sec:examples}

We again assume that 
\begin{empheq}[box=\mybluebox]{equation}
(y,s)\in X\times\RR
\end{empheq}
in this section. 

\subsection{Ice cream cone revisited}

In this subsection, we assume that $\gamma>0$ and that 
\begin{equation}
C = \menge{x\in X}{\|x-z\|\leq\gamma},
\quad
\text{where $\|z\|\leq\gamma$.}
\end{equation}
Then $C$ is obviously (linearly) bounded and hence 
$\rec C=\{0\}$ by \cref{f:rec}\cref{f:rec3}. 
Let $x\in X$. 
We have 
\begin{subequations}
\label{e:211208a}
\begin{align}
P_C(x)&=
\begin{cases}
z+\displaystyle\gamma\frac{x-z}{\|x-z\|}, &\text{if $\|x-z\|>\gamma$;}\\
x, &\text{if $\|x-z\|\leq\gamma$}
\end{cases}\\
&= z + \gamma\cdot \frac{x-z}{\max\{\|x-z\|,\gamma\}}. 
\end{align}
\end{subequations}
Thus
\begin{align}
\label{e:211208b}
x-P_C(x) = \Big(1 - \frac{\gamma}{\max\{\|x-z\|,\gamma\}}\Big) 
(x-z) 
\end{align}
and 
\begin{align}
\label{e:220306a}
\scal{P_C(x)}{x-P_C(x)} 
&= \Big(1 - \frac{\gamma}{\max\{\|x-z\|,\gamma\}}\Big)
\bigg(\scal{x-z}{z}+ \frac{\gamma\|x-z\|^2}{\max\{\|x-z\|,\gamma\}}\bigg).
\end{align}
Moreover, 
$d_C(x) = \max\{0,\|x-z\|-\gamma\}$ 
and thus
\begin{equation}
d^2_C(x) = {\max}^2\{0,\|x-z\|-\gamma\}. 
\end{equation}
Hence, for $\alpha>0$, we have 
\begin{equation}
\label{e:220306b}
\phi(\alpha)= 
\alpha^2d_{C}^2(y/\alpha)
=\alpha^2{\max}^2\{0,\|(y/\alpha)-z\|-\gamma\}
= {\max}^2\{0,\|y-\alpha z\|-\alpha\gamma \}
\end{equation}
and, in view of \cref{e:211207a} and \cref{e:220306a}, 
\begin{subequations}
\begin{align}
\phi'(\alpha) 
&= 
-2\alpha\scal{P_C(y/\alpha)}{(y/\alpha)-P_C(y/\alpha)} \\
&=-2\alpha\bigg(1 - \frac{\gamma}{\max\{\|y/\alpha-z\|,\gamma\}} \bigg)
\bigg(\scal{z}{y/\alpha-z}+\frac{\gamma\|y/\alpha-z\|^2}{\max\{\|y/\alpha-z\|,\gamma\}}\bigg) \label{e:220327a} \\
&= -2\bigg(1 - \frac{\alpha\gamma}{\max\{\|y-\alpha z\|,\alpha\gamma\}} \bigg)
\bigg(\scal{z}{y-\alpha z}+\frac{\gamma\|y-\alpha z\|^2}{\max\{\|y-\alpha z\|,\alpha \gamma\}}\bigg). \label{e:220327c}
\end{align}
\end{subequations}
Because $\Psi(\alpha) = \phi(\alpha)+(\alpha-s)^2$ and hence 
$\Psi'(\alpha) = \phi'(\alpha)+2(\alpha-s)$ when $\alpha>0$, 
we thus have $\Psi(\alpha)$ and $\Psi'(\alpha)$ available to us
for usage in \cref{algo1}.

Let $\alpha^* \in \RP$ be a minimizer of $\Psi$.
If $\alpha^* > 0$, then $\Psi'(\alpha^*) = 0$. Otherwise, if $\alpha^* = 0$, then $\Psi'_+(0) \ge 0$. To compute $\Psi'_+(0)$, we first compute $\lim_{\alpha \to 0^+} \Phi'(\alpha) = \lim_{\alpha \to 0^+} \phi'(\alpha)$ by taking the limit of \cref{e:220327c} as $\alpha \to 0^+$.

We start by assuming that $y = 0$. 
Because $0\in B[z; \gamma]$ by assumption, we have $\|z\| \leq \gamma$. 
Thus, if $\alpha>0$, then 
\begin{equation}
1 - \frac{\alpha\gamma}{\max\{\|0-\alpha z\|,\alpha \gamma\}} = 1 - \frac{\gamma}{\max\{\|0/\alpha - z\|,\gamma\}} = 0,
\end{equation}
which in view of \cref{e:220327c} yields $\phi'(\alpha) = 0$.
So $\Phi'_+(0) = 0$ when $y = 0$.

Now suppose that $y\neq 0$.
Then for all sufficiently small $\alpha>0$, we have 
$\max\{\|y-\alpha z\|,\alpha \gamma\} = \|y - \alpha z\|$. 
Using now \cref{e:220327c} yields
\begin{subequations}
\begin{align}
\Phi'_+(0) &= \lim_{\alpha \to 0^+} -2\bigg(1 - \frac{\alpha\gamma}{\|y-\alpha z\|} \bigg)
\bigg(\scal{z}{y-\alpha z}+\frac{\gamma\|y-\alpha z\|^2}{\|y-\alpha z\|}\bigg) \\
&= -2\langle z, y \rangle - 2\gamma \|y\|. \label{e:220327d}
\end{align}
\end{subequations}

Altogether, in either case, we have
\begin{equation}
\label{e:220530a}
\Phi'_+(0) = -2\langle z, y \rangle - 2\gamma \|y\|.
\end{equation}

Since $\Psi(\alpha) = \Phi(\alpha) + (\alpha - s)^2$ for all $\alpha \ge 0$,
we deduce that 
\begin{equation}
\label{e:220328a}
\Psi'_+(0) = -2s - 2\langle z, y \rangle - 2\gamma \|y\|.
\end{equation}


We now consider further cases.

\emph{Case~1:}
$\Psi_+'(0) = -2s - 2\langle z, y \rangle - 2\gamma \|y\| \geq 0$, i.e., $s + \langle z, y \rangle + \gamma \|y\| \leq 0$. \\
Then $\alpha^*=0$ minimizes $\Psi$
and $P_{\rec C}(y)=0$ because $C$ is bounded. 
Hence \cref{e:P_K} gives 
\begin{equation}
P_K(y,s) = (0,0).
\end{equation}

\emph{Case~2:}
$\Psi_+'(0) = -2s - 2\langle z, y \rangle - 2\gamma \|y\| < 0$, i.e., $s + \langle z, y \rangle + \gamma\|y\|>0$. \\
In this case, the minimizer $\alpha^*$ of $\Psi$ is positive.
To find $\alpha^*$, 
we set the derivative 
$\Psi'(\alpha) = 2(\alpha - s) + \phi'(\alpha)$ equal to $0$. 
Using \cref{e:220327c}, this leads to
\begin{equation}
\label{e:220330a}
\alpha-s = \bigg(1 - \frac{\alpha\gamma}{\max\{\|y-\alpha z\|,\alpha\gamma\}} \bigg)
\bigg(\scal{z}{y-\alpha z}+\frac{\gamma\|y-\alpha z\|^2}{\max\{\|y-\alpha z\|,\alpha \gamma\}}\bigg).
\end{equation}

We consider further cases because of the maximum.

\emph{Case~2.1:} $\|y - \alpha z\|\leq \alpha\gamma$.\\
Then $\max\{\|y - \alpha z\| , \alpha \gamma\} = \alpha \gamma$, 
so $\alpha-s = 0$ and so $\alpha^*=s$. 
Hence the  vector component of $P_K(y,s)$ is 
$\alpha^* P_C(y/\alpha^*)=sP_C(y/s)$
Now $\|y - \alpha^* z\|\leq \alpha^*\gamma$, 
i.e., $\|y - sz\| \leq s\gamma$. 
Thus $\|y/s - z\|\leq \gamma$ and so $y/s\in C$.
Hence $P_C(y/s)=y/s$ and so $sP_C(y/s)=s(y/s) = y$.
Therefore, 
\begin{equation}
P_K(y,s) = (y,s),
\end{equation}
i.e., $(y,s)$ was already in $K$.

\emph{Case~2.2:} $\|y - \alpha z\| > \alpha\gamma$.\\
Then $\max\{\|y - \alpha z\|, \alpha \gamma\} = \|y - \alpha z\|$
and \cref{e:220330a} turns into
\begin{equation}
	\alpha - s = \bigg(1 - \frac{\alpha\gamma}{\|y-\alpha z\|} \bigg)
	\big(\scal{z}{y-\alpha z}+\gamma\|y-\alpha z\|\big);
\end{equation}
equivalently,
\begin{equation}
	\label{e:220408a}
	\big(\alpha \gamma^2 + \alpha - s - \scal{z}{y-\alpha z}\big)\|y - \alpha z\| = -3\alpha \gamma \scal{z}{y} + 2\alpha^2 \gamma \|z\|^2 + \gamma \|y\|^2.
\end{equation}
Squaring both sides, we see $\alpha$ must satisfy
\begin{equation}
	\label{e:220408b}
	\big(\alpha \gamma^2 + \alpha - s - \scal{z}{y-\alpha z}\big)^2\|y - \alpha z\|^2 - \big(-3\alpha \gamma \scal{z}{y} + 2\alpha^2 \gamma \|z\|^2 + \gamma \|y\|^2\big)^2 = 0.
\end{equation}
Expanding \cref{e:220408b} and re-arranging, 
we obtain a quartic equation in $\alpha$ that our optimal value $\alpha^*$ 
must satisfy:
\begin{equation}
	\label{e:220405a}
	\xi_0 + \xi_1 \alpha + \xi_2 \alpha^2 + \xi_3 \alpha^3 + \xi_4 \alpha^4 = 0,
\end{equation}
where
\begin{subequations}
	\label{e:220405b}
	\begin{align}
		\xi_0 &= (s + \scal{z}{y})^2\|y\|^2 - \gamma^2 \|y\|^4 \\
		\xi_1 &= -2(s + \scal{z}{y})(\gamma^2 + 1 + \|z\|^2)\|y\|^2 - 2\scal{y}{z}(s + \scal{z}{y})^2 + 6\gamma^2 \|y\|^2 \scal{z}{y} \\
		\xi_2 &= -4\gamma^2\|y\|^2\|z\|^2 - 9\gamma^2\scal{z}{y}^2 + (\gamma^2 + \|z\|^2 + 1)^2\|y\|^2 \notag\\
		&\quad + (\scal{z}{y} + s)^2\|z\|^2 + 4(\gamma^2 + \|z\|^2 + 1)(\scal{z}{y} + s)\scal{z}{y} \\
		\xi_3 &= 12\gamma^2 \scal{y}{z}\|z\|^2 - 2(\gamma^2 + \|z\|^2 + 1)(\scal{z}{y} + s)\|z\|^2 \notag\\
&\quad		- 2(\gamma^2 + \|z\|^2 + 1)^2\scal{z}{y} \\
		\xi_4 &=  \|z\|^2\big(1+(\gamma+\|z\|)^2\big)\big(1+(\gamma-\|z\|)^2\big)
	\end{align}
\end{subequations}
This is as far as we go with regards to computing an analytic solution for general $z$. In principle, this quartic could be solved analytically by using the usual methods (e.g., Ferrari's method); however, the symbolic and complicated nature of the coefficients seems to be make very hard to obtain the correct solution of the quartic. 
Given the convexity of $\Psi$, we can expect there to be a unique $\alpha > 0$ satisfying \cref{e:220408a}. However, in getting to \cref{e:220408b} and \cref{e:220405a}, the process of squaring both sides may introduce an erroneous solution, even allowing for $\alpha > 0$. So, should an interested reader decide to press onwards and solve \cref{e:220405a} analytically, they may be presented with multiple positive solutions, of which only one can be correct.

For the purpose of making further progress analytically, let us consider the case when
\begin{equation}
z = 0. 
\end{equation}
In this case, \cref{e:220408a} becomes
\begin{equation}
	\big(\alpha \gamma^2 + \alpha - s\big)\|y\| = \gamma \|y\|^2.
\end{equation}
Our assumption that $\|y - \alpha z\| > \alpha \gamma$, i.e., 
$\|y\|>\alpha\gamma$, and that 
$\|z\| \leq \gamma$
preclude the possibility that $y = 0$. 
Thus we obtain 
\begin{equation}
	\label{e:220408c}
	\alpha^* = \frac{s + \gamma \|y\|}{\gamma^2 + 1}
\end{equation}
which we know to be positive.








Now 
$\|y\|>\alpha^*\gamma$
$\Leftrightarrow$
$\|y/\alpha^*\|>\gamma$
$\Leftrightarrow$
$y/\alpha^*\notin C$
$\Leftrightarrow$
\begin{equation}
P_C(y/\alpha^*)= \gamma\frac{y/\alpha^*}{\|y/\alpha^*\|}=\gamma \frac{y}{\|y\|}
\end{equation}
and so 
\begin{equation}
\alpha^* P_C(y/\alpha^*)=
\alpha^* \gamma \frac{y}{\|y\|}
= \gamma\frac{s+\gamma\|y\|}{1+\gamma^2}\frac{y}{\|y\|}.
\end{equation}
We thus conclude that 
\begin{equation}
P_K(y,s) = 
\frac{s+\gamma\|y\|}{1+\gamma^2}\bigg(\frac{\gamma y}{\|y\|}, 1 \bigg). 
\end{equation}
Having exhausted all cases, we summarize the $z = 0$ discussion in the following:

\begin{example}[\bf ice cream cone]
Let $\gamma>0$ and suppose that 
\begin{equation}
C = \menge{x\in X}{\|x\|\leq\gamma}. 
\end{equation}
Then 
\begin{equation}
K = \clcone\big(C\times\{1\}\big) = \bigcup_{\rho\geq 0}\rho\big(C\times\{1\}\big)
\end{equation}
and 
\begin{equation}
\label{e:ice}
P_K(y,s) = 
\begin{cases}
(y,s), &\text{if $\|y\|\leq\gamma s$;}\\[+2mm]
(0,0), &\text{if $\gamma\|y\|\leq -s$;}\\[+2mm]
\displaystyle\frac{s+\gamma\|y\|}{1+\gamma^2}\Big(\frac{\gamma y}{\|y\|}, 1 \Big), &\text{otherwise.}
\end{cases}
\end{equation}
\end{example}
Of course, \cref{e:ice} is consistent with the literature;
see, e.g., \cite[Exercise~29.11]{BC2017}
or \cite[Example~6.37]{Beck} (when $\gamma=1$).

We also present a particular case where $z \neq 0$, both to showcase the complexity in pursuing an analytic solution, and to demonstrate the use of \cref{algo1}.

\begin{example}
\label{ex:bpzneq0}
(See also \cref{fig2a}.)
Suppose that $X = \mathbb{R}^2$ and that 
\begin{equation}
C = \menge{(x_1, x_2) \in X}{(x_1 - 1)^2 + x_2^2 \le 1},
\end{equation}
i.e., $z = (1, 0)$ and $\gamma = 1$. 
Then the projection of $(y, s) = (1, 2, 1)$ onto $K = \clcone\big(C\times\{1\}\big)$ is approximately
$$P_K(y, s) \approx (1.1327162, 1.4226203, 1.4597189).$$

To show this, let us first observe that we are in the difficult Case~2.2: 
$s + \scal{z}{y} + \gamma \|y\| > 0$. 
This implies $(y, s)$ does not belong to $K$, and thus projects 
onto $\cone(C \times 1)$.
Using \cref{e:220405a} with our particular $z$, $y$, $\gamma$, and $s$, we obtain
\begin{equation}
\label{e:220512}
5 \alpha^4 - 18 \alpha^3 + 44 \alpha^2 - 38 \alpha - 5 = 0,
\end{equation}
the left hand side of which is irreducible over $\mathbb{Q}$ 
For this reason, we turn to \cref{algo1}. First, we explicitly compute $\Psi'(\alpha) = \phi'(\alpha) + 2(\alpha - s)$, using \cref{e:220327c}:

\begin{equation}
\Psi'(\alpha) = -2\bigg(1 - \frac{\alpha}{\max\{\sqrt{(\alpha - 1)^2 + 4},\alpha\}} \bigg)
\bigg(1-\alpha+\frac{(\alpha - 1)^2 + 4}{\max\{\sqrt{(\alpha - 1)^2 + 4},\alpha\}}\bigg) + 2(\alpha - 1).
\end{equation}

We apply \cref{algo1} to find the optimal $\alpha^* > 0$, satisfying $\Psi'(\alpha^*) = 0$, and thus \cref{e:220512}. 
We do this with the starting points $\alpha = 3$, $\beta = 5$, and the error tolerance $\epsilon = 10^{-6}$. The results are recorded in \cref{table1}. 
The algorithm terminates with $\alpha^* = 1.4597189$; 
thus, we conclude that 
\begin{equation}
P_K(1, 2, 1) \approx (\alpha^* P_C(y/\alpha^*), \alpha^*) \approx (1.1327162, 1.4226203, 1.4597189).
\end{equation}

\begin{table}[ht!]
\centering
\begin{tabular}{c|ccc|ccc}
$n$ & $\alpha$ & $\frac{\alpha + \beta}{2}$ & $\beta$ & $\Psi'(\alpha)$ & $\Psi'\left(\frac{\alpha + \beta}{2}\right)$ & $\Psi'(\beta)$ \rule[-0.9ex]{0pt}{0pt} \\
\hline
$ 1 $ & $ 3.0000000 $ & $ - $ & $ 5.0000000 $ & $ 4.10 \times 10^{0} $ & $ - $ & $ 8.11 \times 10^{0} $ \rule{0pt}{2.6ex} \\
$ 2 $ & $ 1.5000000 $ & $ - $ & $ 3.0000000 $ & $ 1.49 \times 10^{-1} $ & $ - $ & $ 4.10 \times 10^{0} $ \\
$ 3 $ & $ 0.75000000 $ & $ 1.1250000 $ & $ 1.5000000 $ & $ -3.35 \times 10^{0} $ & $ -1.310 \times 10^{0} $ & $ 1.49 \times 10^{-1} $ \\
$ 4 $ & $ 1.1250000 $ & $ 1.3125000 $ & $ 1.5000000 $ & $ -1.310 \times 10^{0} $ & $ -5.79 \times 10^{-1} $ & $ 1.49 \times 10^{-1} $ \\
$ 5 $ & $ 1.3125000 $ & $ 1.4062500 $ & $ 1.5000000 $ & $ -5.79 \times 10^{-1} $ & $ -2.04 \times 10^{-1} $ & $ 1.49 \times 10^{-1} $ \\
$ 6 $ & $ 1.4062500 $ & $ 1.4531250 $ & $ 1.5000000 $ & $ -2.04 \times 10^{-1} $ & $ -2.48 \times 10^{-2} $ & $ 1.49 \times 10^{-1} $ \\
$ 7 $ & $ 1.4531250 $ & $ 1.4765625 $ & $ 1.5000000 $ & $ -2.48 \times 10^{-2} $ & $ 6.29 \times 10^{-2} $ & $ 1.49 \times 10^{-1} $ \\
$ 8 $ & $ 1.4531250 $ & $ 1.4648438 $ & $ 1.4765625 $ & $ -2.48 \times 10^{-2} $ & $ 1.92 \times 10^{-2} $ & $ 6.29 \times 10^{-2} $ \\
$ 9 $ & $ 1.4531250 $ & $ 1.4589844 $ & $ 1.4648438 $ & $ -2.48 \times 10^{-2} $ & $ -2.76 \times 10^{-3} $ & $ 1.92 \times 10^{-2} $ \\
$ 10 $ & $ 1.4589844 $ & $ 1.4619141 $ & $ 1.4648438 $ & $ -2.76 \times 10^{-3} $ & $ 8.23 \times 10^{-3} $ & $ 1.92 \times 10^{-2} $ \\
$ 11 $ & $ 1.4589844 $ & $ 1.4604492 $ & $ 1.4619141 $ & $ -2.76 \times 10^{-3} $ & $ 2.74 \times 10^{-3} $ & $ 8.23 \times 10^{-3} $ \\
$ 12 $ & $ 1.4589844 $ & $ 1.4597168 $ & $ 1.4604492 $ & $ -2.76 \times 10^{-3} $ & $ -1.06 \times 10^{-5} $ & $ 2.74 \times 10^{-3} $ \\
$ 13 $ & $ 1.4597168 $ & $ 1.4600830 $ & $ 1.4604492 $ & $ -1.06 \times 10^{-5} $ & $ 1.36 \times 10^{-3} $ & $ 2.74 \times 10^{-3} $ \\
$ 14 $ & $ 1.4597168 $ & $ 1.4598999 $ & $ 1.4600830 $ & $ -1.06 \times 10^{-5} $ & $ 6.77 \times 10^{-4} $ & $ 1.36 \times 10^{-3} $ \\
$ 15 $ & $ 1.4597168 $ & $ 1.4598083 $ & $ 1.4598999 $ & $ -1.06 \times 10^{-5} $ & $ 3.33 \times 10^{-4} $ & $ 6.77 \times 10^{-4} $ \\
$ 16 $ & $ 1.4597168 $ & $ 1.4597626 $ & $ 1.4598083 $ & $ -1.06 \times 10^{-5} $ & $ 1.61 \times 10^{-4} $ & $ 3.33 \times 10^{-4} $ \\
$ 17 $ & $ 1.4597168 $ & $ 1.4597397 $ & $ 1.4597626 $ & $ -1.06 \times 10^{-5} $ & $ 7.53 \times 10^{-5} $ & $ 1.61 \times 10^{-4} $ \\
$ 18 $ & $ 1.4597168 $ & $ 1.4597282 $ & $ 1.4597397 $ & $ -1.06 \times 10^{-5} $ & $ 3.24 \times 10^{-5} $ & $ 7.53 \times 10^{-5} $ \\
$ 19 $ & $ 1.4597168 $ & $ 1.4597225 $ & $ 1.4597282 $ & $ -1.06 \times 10^{-5} $ & $ 1.09 \times 10^{-5} $ & $ 3.24 \times 10^{-5} $ \\
$ 20 $ & $ 1.4597168 $ & $ 1.4597197 $ & $ 1.4597225 $ & $ -1.06 \times 10^{-5} $ & $ 1.65 \times 10^{-7} $ & $ 1.09 \times 10^{-5} $ \\
$ 21 $ & $ 1.4597168 $ & $ 1.4597182 $ & $ 1.4597197 $ & $ -1.06 \times 10^{-5} $ & $ -5.20 \times 10^{-6} $ & $ 1.65 \times 10^{-7} $ \\
$ 22 $ & $ 1.4597182 $ & $ 1.4597189 $ & $ 1.4597197 $ & $ -5.20 \times 10^{-6} $ & $ -2.52 \times 10^{-6} $ & $ 1.65 \times 10^{-7} $ \\
$ 23 $ & $ \underline{1.4597189} $ & $ 1.4597193 $ & $ 1.4597197 $ & $ -2.52 \times 10^{-6} $ & $ -1.18 \times 10^{-6} $ & $ 1.65 \times 10^{-7} $
\end{tabular}
\caption{Behavior of \cref{algo1} for \cref{ex:bpzneq0}}
\label{table1}
\end{table}
\end{example}

\begin{figure}
\centering
\includegraphics[width=0.75\textwidth]{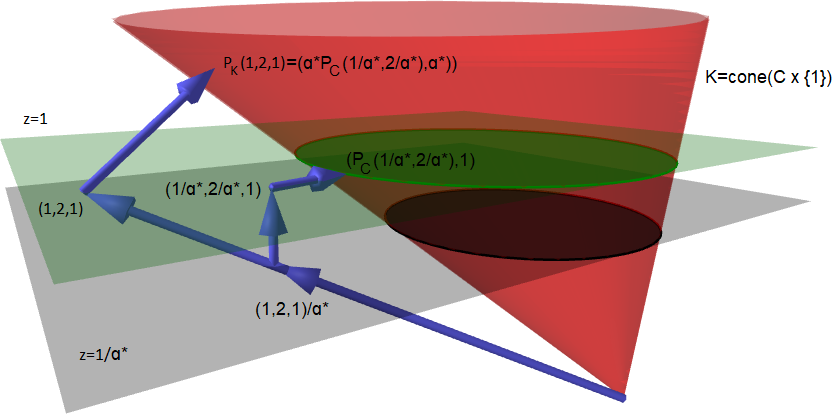}
\caption{Illustration of \cref{ex:bpzneq0}}
\label{fig2a}
\end{figure}

\subsection{Abstract ball-pen set}

In this subsection, we assume that $B$ is the closed unit ball centred at $0$,  
$R$ is a nonempty closed convex cone, and 
\begin{equation}
C = B+R.
\end{equation}
Then $\rec C = R$, 
\begin{equation}
P_C(x) = 
\begin{cases}
x, &\text{if $d_R(x)\leq 1$;}\\
\displaystyle P_Rx+\frac{x-P_Rx}{d_R(x)}, &\text{if $d_R(x)>1$,}
\end{cases}
\end{equation}
and hence 
\begin{equation}
d_C(x) = \max\{0,d_R(x)-1\}.
\end{equation}
Because $R$ is a cone, we have $d_R(y/\alpha)=\tfrac{1}{\alpha}d_R(y)$ and thus
\begin{equation}
\alpha^2d^2_C(y/\alpha) = {\max}^2\big\{0,d_R(y)-\alpha \big\}. 
\end{equation}
Hence \cref{e:Psi} turns into
\begin{equation}
\Psi(\alpha) = 
\begin{cases}
\pinf, &\text{if $\alpha<0$;}\\
d^2_{R}(y)+s^2, &\text{if $\alpha=0$;}\\
{\max}^2\big\{0,d_R(y)-\alpha \big\}+(\alpha-s)^2, &\text{if $\alpha>0$.}
\end{cases}
\end{equation}
Now if $\alpha>0$, then $\Psi'(\alpha)= -2\max\{0,d_R(y)-\alpha\} + 2(\alpha-s)$ and thus 
\begin{equation}
(\forall \alpha \geq 0)\quad 
\Psi'_+(\alpha) = 
2\min\{\alpha-s,2\alpha-d_R(y)-s\}. 
\end{equation}
Hence $\Psi'_+(0)=2\min\{-s,-d_R(y)-s\} = -2(d_R(y)+s)$.
If $d_R(y)\leq -s$, then $\Psi'_+(0)\geq 0$; 
consequently, $\alpha^*=0$ and 
\begin{equation}
P_K(y,s) = (P_Ry,0). 
\end{equation}
Otherwise, $d_R(y)>-s$ $\Leftrightarrow$ $\alpha^*>0$
$\Leftrightarrow$ 
$2\min\{\alpha^*-s,2\alpha^*-d_R(y)-s\}=0$
$\Leftrightarrow$ 
$\min\{\alpha^*,2\alpha^*-d_R(y)\}=s$
$\Leftrightarrow$ 
\begin{equation}
\alpha^* = 
\begin{cases}
s, &\text{if $d_R(y)\leq s$;}\\
\displaystyle \frac{s+d_R(y)}{2}, &\text{if $d_R(y)>s$.}
\end{cases}
\end{equation}
To sum up, 
\begin{equation}
\alpha^* = 
\begin{cases}
0, &\text{if $d_R(y)\leq -s$;}\\
s, &\text{if $-s<d_R(y)\leq s$;}\\
\displaystyle \frac{s+d_R(y)}{2}, &\text{if $|s|<d_R(y)$}
\end{cases}
\end{equation}
and thus 
\begin{equation}
P_K(y,s) = 
\begin{cases}
(P_R(y),0), &\text{if $d_R(y)\leq -s$;}\\
\big(sP_C(y/s),s\big), &\text{if $-s<d_R(y)\leq s$;}\\
\Big(\frac{s+d_R(y)}{2}P_C\big(2y/(s+d_R(y)\big), \frac{s+d_R(y)}{2} \Big), &\text{if $|s|<d_R(y)$.}
\end{cases}
\end{equation}

\begin{example}[\bf classical ball-pen set]
\label{eg:classicalballpen}
The classical ball-pen example arises when 
$X=\RR^2$ and $R = \RR_+(0,1)$. 
Then $P_Rx = \max\{0,x_2\}(0,1)=(0,\max\{0,x_2\})$ and 
$d_R^2(x)= \|x\|^2 - \max^2\{0,x_2\}=x_1^2+\min^2\{0,x_2\}$. 
Setting 
\begin{equation}
\delta(y) := d_R(y) = \tsqrt{y_1^2+\min^2\{0,y_2\}},
\end{equation}
we deduce that 
\begin{equation}
P_K(y,s) = P_K(y_1,y_2,s) = 
\begin{cases}
\big(0,\max\{0,y_2\},0\big), &\text{if $\delta(y)\leq -s$;}\\
\big(sP_C(y/s),s\big), &\text{if $-s<\delta(y)\leq s$;}\\
\Big(\frac{s+\delta(y)}{2}P_C\big(2y/(s+\delta(y)\big), \frac{s+\delta(y)}{2} \Big), &\text{if $|s|<\delta(y)$.}
\end{cases}
\end{equation}
Note that $(y_1,y_2,s)\in K^\ominus$
$\Leftrightarrow$
$P_K(y_1,y_2,s)=(0,0,0)$
$\Leftrightarrow$
$y_2 \leq 0$ and $\sqrt{y_1^2+y_2^2}+s \leq 0$; 
consequently,
\begin{equation}
K^\ominus = \menge{(y_1,y_2,y_3)\in\RR\times\RR_-\times\RR_-}{\sqrt{y_1^2+y_2^2}+y_3\leq 0}. 
\end{equation}
\end{example}

\section*{Acknowledgments}
The authors thank Professors Simeon Reich and Tadeusz Kuczumow for 
comments concerning linearly bounded sets and providing us with a copy of 
\cite{GK1978}; 
Professor Daniel Klain for comments on polar sets and pointing us to 
\cite{Schneider} as a reference; 
and Professor Levent Tun\c{c}el for comments on the homogenization cone and 
providing us with \cite{RoshchinaTuncel}.
HHB is supported by the Natural Sciences and
Engineering Research Council of Canada.


\appendix

\section{Appendix}

\label{app:A}

\begin{example}
\label{app:1}
Suppose that $X=\RR^n$ and 
\begin{equation}
C = \menge{(x_1,\ldots,x_n)\in \RR^n}{x_1+\cdots+x_n\leq 1,\;\text{each } x_i\geq 0}.
\end{equation}
Then
\begin{equation}
\label{e:220311a}
C^\odot = \menge{(y_1,\ldots,y_n)\in \RR^n}{\text{each } y_i\leq 1}.
\end{equation}
\end{example}
\begin{proof}
Let $x\in X$ and $y\in X$. 
Denote the right side of \cref{e:220311a} by $D$.
Suppose first that $x\in C$ and $y\in D$.
Then 
\begin{equation}
\scal{x}{y} = \sum_{i=1}^n x_iy_i 
\leq \sum_{i=1}^n x_i  \leq 1. 
\end{equation}
Hence $D\subseteq C^\odot$.

Now assume that $y\in X\smallsetminus D$.
Then there exists an index such that $y_i>1$. 
Suppose that $x_i=1$ and $x_j=0$ when $j\neq i$. 
Then $x\in C$ and 
$\scal{x}{y} = y_i>1$ and so $y\notin C^\odot$. 
Hence $X\smallsetminus D \subseteq X\smallsetminus C^\odot$. 

Altogether, we have verified \cref{e:220311a}. 
\end{proof}

\begin{example}
\label{app:3}
Suppose that $X=\RR^2$ and 
\begin{equation}
C = \Menge{(x_1,x_2)\in \RR^2}{x_1\leq 1 - \sqrt{1+x_2^2}}.
\end{equation}
Then
\begin{equation}
\label{e:220315a}
\sigma_C(y_1,y_2)
= \begin{cases}
+\infty, &\text{if $y_1<|y_2|$;}\\
y_1-\sqrt{y_1^2-y_2^2}, &\text{if $|y_2|\leq y_1$},
\end{cases}
\end{equation}
and 
\begin{subequations}
\begin{align}
\label{e:220311c}
C^\odot &= 
\menge{(y_1,y_2)\in\RR^2}{|y_2|\leq y_1 \;\text{and}\; 1+y_2^2\leq 2y_1}\\
&=
\conv \Big(\{0\}\cup \Menge{(y_1,y_2)\in \RR^2}{y_1\geq (1+y_2^2)/2}\Big).
\end{align}
\end{subequations}
\end{example}
\begin{proof}
We shall compute $\sigma_C$.
Let us write, for simplicity, $(x_1,x_2)=(x,y)$. 
Then $(x,y)\in C$ $\Leftrightarrow$ $x\leq 1- \sqrt{1+y^2}\leq 0$, so for sure $x\leq 0$.
Next, $\sigma_C$ is fully determined by the boundary of $C$, which is given by 
\begin{equation}
x = 1-\tsqrt{1+y^2}.
\end{equation}
Let $(u,v)\in\RR^2$. 
Then 
\begin{equation}
\sigma_C(u,v) = \sup_{y\in\RR}\Big(u\big(1-\tsqrt{1+y^2}\big) +vy\Big). 
\end{equation}
Note that $\sigma_C(u,v)=\sigma_C(u,-v)$, so we may and do assume that $v\geq 0$. 
Moreover, if $u<0$, then clearly $\sigma_C(u,v)=+\infty$.
And if $u=0$, then $\sigma_C(u,v)=\sigma_C(0,v) = \sup_{y\in\RR}vy = \iota_{\{0\}}(v)$
consistent with our announced formula. 
So we assume from now on that also $u> 0$. 

\emph{Case~1:} $0<u<v$.\\
Because 
\begin{subequations}
\begin{align}
u\big(1-\tsqrt{1+y^2}\big) +vy
&= u+(v-u)\tsqrt{1+y^2} - \displaystyle\frac{v}{y+\tsqrt{1+y^2}}\\
&\to u+\infty - 0 
\end{align}
\end{subequations}
as $y\to+\infty$, we see that $\sigma_C(u,v)=+\infty$ which is consistent 
with our announced formula. 

\emph{Case~2:} $0<u=v$.\\
Because 
\begin{subequations}
\begin{align}
v\big(1-\tsqrt{1+y^2}\big) +vy
&= v - \displaystyle\frac{v}{y+\tsqrt{1+y^2}}\\
&\to v
\end{align}
\end{subequations}
as $y\to+\infty$, we see that $\sigma_C(u,v)=v$ which is consistent 
with our announced formula. 

\emph{Case~3:} $0<v<u$.\\
Here we can use Calculus!
Note that $0 = \tfrac{d}{dy} u(1-\tsqrt{1+y^2}) +vy = v-uy/\tsqrt{1+y^2}$ 
$\Leftrightarrow$ 
\begin{equation}
\frac{y}{\tsqrt{1+y^2}} = \frac{v}{u} =: q < 1.
\end{equation}
This equation can be solved for $y$; indeed, the solution is 
\begin{equation}
y := \frac{q}{\tsqrt{1-q^2}} = \frac{v}{\tsqrt{u^2-v^2}}. 
\end{equation}
One checks that $\sqrt{1+y^2} = u/\sqrt{u^2-v^2}$ and hence
\begin{align}
\sigma_C(u,v)
&=
u\big(1-\tsqrt{1+y^2}\big)+vy \\
&= u-\frac{u^2-v^2}{\sqrt{u^2-v^2}}\\
&= u -\sqrt{u^2-v^2}.
\end{align}

The three cases considered verify \cref{e:220315a}.
Now suppose that $|v|\leq u$. 
Then $(u,v)\in C^\odot$
$\Leftrightarrow$
$\sigma_C(u,v)\leq 1$
$\Leftrightarrow$
$u-\tsqrt{u^2-v^2}\leq 1$
$\Leftrightarrow$
$u-1\leq \sqrt{u^2-v^2}$. 

If $u\leq 1$, then $u-1\leq 0 \leq \sqrt{u^2-v^2}$ and thus $(u,v)\in C^\odot$. 
And if 
$u\geq 1$, then 
$u-1\leq \sqrt{u^2-v^2}$
$\Leftrightarrow$
$(u-1)^2 \leq u^2-v^2$
$\Leftrightarrow$
$1+v^2\leq 2u$.
\end{proof}

\end{document}